\documentclass[12pt,leqno,a4paper]{amsart}
\usepackage{amssymb}  
\usepackage{amsmath}  
\usepackage{amsthm}  
\usepackage{verbatim} 
\theoremstyle{plain}
\newtheorem{theorem}{THEOREM}[section]
\newtheorem{proposition}{PROPOSITION}[section]
\newtheorem{lemma}{LEMMA}[section]
\newtheorem{corollary}{COROLLARY}[section]

\theoremstyle{definition}
\newtheorem{definition}{DEFINITION}[section]

\newtheorem{remark}{Remark}[section]

\numberwithin{equation}{section}
\newcommand{\ga}{\alpha}

\newcommand{\gd}{\delta}
\newcommand{\gD}{\Delta}
\newcommand{\gep}{\varepsilon}

\newcommand{\gth}{\vartheta}
\newcommand{\gTh}{\Theta}

\newcommand{\gl}{\lambda}
\newcommand{\gL}{\Lambda}
\newcommand{\gm}{\mu}
\newcommand{\gn}{\nu}

\newcommand{\gp}{\pi}

\newcommand{\gr}{\rho}
\newcommand{\gs}{\sigma}

\newcommand{\gf}{\varphi}
\newcommand{\gF}{\varPhi}

\newcommand{\gPs}{\Psi}

\newcommand{\gO}{\Omega}
\newcommand{\ov}{\overline}
\newcommand{\de}{\partial}

\newcommand{\R}{\mathbb {R}}

\newcommand{\N}{\mathbb {N}}
\newcommand{\ldue}{L^2}
\newcommand{\Ldue}{\textbf{L}^2 (\gO)}
\newcommand{\Lp}{\textbf{L}^p (\gO)}
\newcommand{\Linf}{\textbf{L}^{\infty} (\gO)}

\newcommand{\huno}{H_0^1(\Omega \cup \Gamma)}
\newcommand{\Huno}{\textbf{H}_0^1 (\gO)}

\newcommand{\Rm}{\R ^m}
\newcommand{\noi}{\noindent}
\newcommand{\be}{\begin{equation}}
\newcommand{\ee}{\end{equation}}

\begin{document}

\title[Sectional symmetry for systems in cylindrical domains]
{Sectional symmetry of solutions of elliptic  systems in cylindrical domains}

\author[Damascelli]{Lucio Damascelli}
\address{ Dipartimento di Matematica, Universit\`a  di Roma
" Tor Vergata " - Via della Ricerca Scientifica 1 - 00173  Roma - Italy.}
\email{damascel@mat.uniroma2.it}
\author[Pacella]{Filomena Pacella}
\address{Dipartimento di Matematica, Universit\`a di Roma
"Sapienza" -  P.le A. Moro 2 - 00185 Roma - Italy.}
\email{pacella@mat.uniroma1.it}
\date{}
\thanks{Supported by PRIN-2009-WRJ3W7 grant}
\begin{abstract} In this paper we prove a kind of rotational symmetry for solutions of semilinear elliptic systems in some bounded cylindrical domains.
The symmetry theorems  obtained hold for low-Morse index solutions whenever the  nonlinearities satisfy some convexity assumptions. These results extend and improve those obtained in 
\cite{DaPaSys, DaGlPa1, Pa, PaWe}.
 \end{abstract}

\maketitle

\section{Introduction}
We consider the  Dirichlet problem for a  semilinear elliptic system of the type
 \begin{equation} \label{genprobsys} 
\begin{cases} - \gD U =F(x,U) \quad &\text{in } \gO \\
U= 0 \quad & \text{on } \de  \gO
\end{cases}
\end{equation}

\noi  where  $\Omega $ is a smooth bounded domain  in $\R^N$, $N \geq 2$ 
 and  $F=(f_1 , \dots , f_m)$ is a function belonging to $C^{1 }( \overline {\Omega} \times \Rm; \Rm)$, $m \geq 1$. 
 Here $U=(u_1, \dots ,u_m)$ is a vector valued function. \par
When  $m=1$, i.e.   the  equation  in \eqref{genprobsys}  is a scalar semilinear elliptic equation, the famous symmetry result by  B. Gidas, W.M. Ni and L. Nirenberg  \cite{GiNiNi}, based on the moving planes method,  asserts that if $\Omega $ is a ball then every positive solution of \eqref{genprobsys} is radial if the nonlinear term $f=f(|x|,u)$ is monotone decreasing with respect to  $r=|x|$.
The result of \cite{GiNiNi} was then extended to systems in \cite{Tr}, \cite{deF1}, \cite{deF2}. \par
 It is well known that the radial symmetry of a solution does not hold, in general, when $\Omega $ is an annulus or if sign changing solutions are considered and even if $f$ does not have the right monotonicity with respect to $|x|$ (see for example \cite{PaRa}).
Nevertheless when the hypotheses of the theorem of Gidas, Ni and Nirenberg fail another kind of symmetry can be recovered, namely the foliated Schwarz symmetry for solutions of \eqref{genprobsys}  in a  ball or in annulus having low Morse index and assuming that the nonlinear term has some convexity properties in the $U$-variable. We refer to Section 2 for the definition of Morse index. \par
A symmetry result of this type was first proved in \cite{Pa} in the case $m=1$  for solutions having Morse index one and assuming that the nonlinearity $f=f(|x|,s)$ is convex in the second variable. Later it was extended in \cite{PaWe} to solutions having Morse index not larger than the dimension $N$ and assuming that the derivative $\frac {\partial f } {\partial s}$ is a convex function in the $s$-variable.
Finally in \cite{DaPaSys} and \cite{DaGlPa1} the foliated Schwarz symmetry  was proved for low Morse index solutions of cooperative elliptic systems, i.e. when $m \geq 2$.
Let us point out that the extension of the results in \cite{Pa} and \cite{PaWe} to systems is nontrivial. Indeed the results of   \cite{DaPaSys} could not be proved for any convex nonlinearity $F=F(|x|,U) $ but  some additional hypotheses were required. \par
In this paper we extend the above results by considering more general symmetric domains and not just balls or annulus. As a consequence we will get less symmetry of the solutions, depending also on a tighter bound on their Morse index. Moreover we are able to improve the results in \cite{DaPaSys}  by allowing any convex nonlinearity in \eqref{genprobsys}. \par
To state precisely our results we need some preliminary definitions. The first one concerns the domains we consider.

  Let  $N \geq 2 $  and $2 \leq k \leq N$.  If $k <N$, let us denote  by $x=(x', x'' )$ a point in $ \R^N$, with $x' \in \R^k$, $x'' \in \R^{N-k}$, 
     and for a bounded domain $\Omega $ let us denote  by $\Omega ''  $  the set 
 $$
\Omega '' = \{ x'' \in \R^{N-k}: \exists \; x' \in \R^k : (x', x'') \in \Omega  \}
$$
   We will consider domains of the following type.
\begin{definition}\label{DefDominiSimmetrici} Assume that $N \geq 2$, $2 \leq k \leq N$.
We say that a bounded domain $\Omega $ in $\R^N$, is  $k$-rotationally symmetric if either $k=N $ and $\Omega $ is a ball or an annulus, or $2 \leq k<N$ and the sets 
 $$\Omega ^{h}= \Omega   \cap \{ x=(x',x'') \in \R^N : x''=h \} 
 $$
 are   either  $k$-dimensional balls or    $k$-dimensional annulus  with the center on  $( 0 ,  h )$
  for every  $ h \in \Omega ''$.
 \end{definition} 
 In other words we require that the set $\Omega ^h$, which represents a section of $\Omega  $ at the level $x''=h $ is either a ball or an annulus in dimension $k$. \par
   The symmetry we will get for solutions of \eqref{genprobsys}   when $\Omega $ is   $k$-rotationally symmetric   is a variant of the  foliated Schwarz symmetry considered in several previous paper
   (see    \cite{BaWi},  \cite{DaPaSys},  \cite{DaGlPa1},  \cite{DaGlPa2},  \cite{GlPaWe},  \cite{Pa},   \cite{PaWe}, \cite{SmWi}, \cite{WethSurvey} and the references therein). \par
 We will call it $k$-sectional foliated Schwarz symmetry. 

 \begin{definition}\label{sectionfoliatedSS-Sys}
 Let $\Omega $ be a bounded  $k$-rotationally symmetric domain   in $\R^N$,   $2 \leq k \leq N$, and let $U: \Omega \to \R ^m $ a continuous function.
 We say that $U$ is \emph{$k$-sectionally foliated Schwarz symmetric} if there exists a vector $p'= (p_1, \dots , p_{k}, 0, \dots ,0) \in \R^N$,  $|p'|=1$, such that 
 $U(x)=U(x', x'' )$ depends only on $x''$, $r= |x'|$ and $\gth = \arccos (\frac {x'}{|x'|} \cdot {p'})$ and $U$ is nonincreasing in $\gth$.
 \end{definition}
 \par
 \smallskip
 When $k=N$ the previous definition coincides with that of foliated Schwarz symmetry. 
Definition \ref{sectionfoliatedSS-Sys} just means that the functions $x' \mapsto U(x', h)$ defined in $\Omega ^h$ are either radial for any $h \in \Omega '' $, or nonradial but foliated Schwarz symmetric for any 
 $h \in \Omega '' $, with the same axis of symmetry. In the case $k=N-1$ the sectional  foliated Schwarz symmetry was defined in \cite{DaPaMixed} to study some elliptic problems with nonlinear mixed boundary conditions.
 \par
 In order to prove symmetry of solutions we also need some symmetry on the nonlinearity. Therefore from now on we assume that $\Omega $ is a smooth bounded  $k$-rotationally symmetric   domain  in $\R^N$ and we rewrite the system  \eqref{genprobsys} as 
 \begin{equation} \label{modprob} 
\begin{cases} - \gD U =F(|x'|, x'',U) \quad &\text{in } \gO \\
U= 0 \quad & \text{on } \de  \gO
\end{cases}
\end{equation}
i.e. we require that $F$ depends radially on $x'$. As in \eqref{genprobsys}  
$F=F(r,x'',S)=(f_1 (r,x'',S), \dots , f_m(r,x'',S) )$ satisfies
\be \label{ConditionsOnF}
F \in C^{1 }([0,+\infty )\times \Omega '' \times \R^m; \R^m)
\ee
\par
\smallskip

 The symmetry results we get are the following (the definition of Morse index and fully coupled  systems will be  recalled in Section 2).
 
 \begin{theorem} \label{fconvessaSistemi} Let $\gO$ be a $k$-rotationally symmetric  domain in $\R^N$, $2 \leq k \leq N  $, and let $U \in C^{2 }(\ov {\gO} ; \R^m)$ be a solution of \eqref{modprob} with $F$ satisfying \eqref{ConditionsOnF}. Assume that  
 \begin{itemize} 
\item [i) ]  the system \eqref{modprob}  is fully coupled along $U$ in $\gO$
 \item [ii) ] for any $i=1, \dots m $    the scalar function $f_i  (|x'|, x'',S) $ is convex in the variable
  $S=(s_1, \dots , s_m) \in \R^m $. 
  \end{itemize}
If   \  $m (U) \leq k $, where  $m (U) $ is the Morse index of $U$, then
 $U$ is  $k$-sectional foliated Schwarz symmetric, and  if the functions  $x' \mapsto U(x', h)$, $h \in \Omega ''$,  are not radial then they are strictly decreasing in the angular variable.
\end{theorem} 
 
 This theorem not only extends the results in \cite{DaPaSys} to $k$-rotationally symmetric  domains but also improves the result of  \cite{DaPaSys} for the case $k=N$ since it only requires that the components $f_i$ of the nonlinearity are convex without further assumptions.
  \par
  The next theorem concerns the case  when the nonlinearity has convex first derivatives.
   
 \begin{theorem} \label{f'convessaSistemi} Let $\gO$ be a $k$-rotationally symmetric in $\R^N$, $2  \leq k \leq N  $, and let $U \in C^{2 }(\ov {\gO}; \R^m)$ be a solution of \eqref{modprob}.
  Assume that:
\begin{itemize} 
\item [i) ]  the system  \eqref{modprob}  is fully coupled along $U$ in $\gO $  
\item [ii) ] for any $i,j=1, \dots m $  the function   $\frac {\de f_i} {\de s_j}(|x'|, x'',S)$ is convex in $S=(s_1, \dots , s_m)$.
  \end{itemize}
 If $m (U) \leq k-1 $ then a solution  $U$ is  $k$-sectionally foliated Schwarz symmetric and  if the functions  $x' \mapsto U(x', h)$, $h \in \Omega ''$,  are not radial then they are strictly decreasing in the angular variable.\end{theorem} 
 
 The previous theorem extends to $k$-rotationally symmetric domains the result in \cite{DaGlPa1} and we provide a different proof which also simplify the one given in  \cite{DaGlPa1}.  Note that in Theorem \ref{f'convessaSistemi} the bound on the Morse index $m(U) \leq k-1$ is stricter than in Theorem \ref{fconvessaSistemi}. When $m=1$, i.e. in the scalar case, it is possible to improve it to $m (U) \leq k$ (adapting the proof  in \cite{PaWe} for $k=N$). However in the vectorial case serious difficulties arise when $m(U)=k$, which prevent to use the same approach, though we believe that the symmetry result should be true also  in this case.\par
 It is interesting to see in the previous theorems how the Morse index of a solution is related to the ''dimension''  of the sectional symmetry of the domain.

 \begin{remark} \label{SoluzioniStabili}
  In the particular case of stable solutions (see Section 2 for the definition) we get the radial symmetry on each section $\Omega ^h $ without requiring any convexity on $F$. This can be proved easily as in the proof of Theorem 1.5 of \cite{GlPaWe}.
  \end{remark}
      We will deduce from  the proof of Theorem \ref{fconvessaSistemi} and Theorem \ref{f'convessaSistemi}  that for nonradial Morse index one solutions the following condition holds.
 
 \begin{corollary} \label{corollario1}
 Under the assumptions of Theorem \ref{fconvessaSistemi} or Theorem \ref{f'convessaSistemi}  if a solution $U$  has Morse index one and  is not radial then  
\be \label{superfullycoupling1}
 \sum_{j=1}^m  \frac {\de f _i}{\de s _j}(r,x'', U(r, \gth)) \frac {\de u_j } {\de \gth  }(r,x'', \gth )=
 \sum_{j=1}^m  \frac {\de f _j}{\de s _i}(r,x'', U(r, \gth))   \frac {\de u_j } {\de \gth  }(r,x'', \gth )
\ee
 for any    $i=1,\dots ,m $, with $(r, \gth)$ as in Definition \ref{sectionfoliatedSS-Sys}. \\ 
   In particular  if $m=2$ then \eqref{superfullycoupling1} implies that 
 \be \label{superfullycoupling2} 
 \frac {\de f _1}{\de s _2}(|x'|,x'', U(x))=  \frac {\de f _2}{\de s _1}(|x'|,x'',  U(x)) \; , \quad \forall \, x \in \gO \;
 \ee
  \end{corollary}
  \begin{remark}
  Under the assumptions of Theorem \ref{f'convessaSistemi}, the conditions   \eqref{superfullycoupling1} and \eqref{superfullycoupling2} hold more generally for solutions having Morse index $m(U) \leq k-1$ if 
   for some $ i_0, j_0 \in \{1, \dots ,m  \} $  the function $\frac {\de f_{i_0}} {\de s_{j_0}}(|x|,S)$ satisfies a strict convexity assumption as in Theorem 1.3 in \cite{DaGlPa1}.
  \end{remark}

   The paper is organized as follows. \par
 In Section 2 we recall suitable versions of weak and strong maximum principles as well as comparison principles for systems. Moreover we state some results from
 the spectral theory for an eigenvalue problem  related to a symmetrized version of the system \eqref{genprobsys}. Finally we define the Morse index. 
In Section 3 we give some sufficient conditions for  $k$-sectional foliated Schwarz symmetry and  prove  Theorem \ref{fconvessaSistemi}, Theorem \ref{f'convessaSistemi} and Corollary \ref{corollario1}.
 \par
 \medskip

\section{Preliminaries}
\subsection{Spectral theory for linear elliptic systems} \ \\
 Let $\gO$ be a  bounded domain in $\R^N$, $N \geq 2$, and $D$  a $m \times m $ matrix with bounded entries:
\begin{equation}\label{ipotesiD} D= \left (  d_{ij} \right ) _{i,j=1}^m \; , \; d_{ij} \in L^{\infty} (\gO)
\end{equation}
We consider the linear elliptic  system

\begin{equation}  \label{linearsystem} 
\begin{cases} - \gD U + D(x) U = F \quad &\text{in } \gO \\
U= 0 \quad & \text{on } \de  \gO
\end{cases}
\end{equation}

i.e. 

$$
\begin{cases}
 - \gD u_1 +  d_{11} u_1+  \dots   +d_{1m} u_m =f_1  & \text{ in } \gO  \\
  \dots  \dots & \dots \\ 
 - \gD u_m +  d_{m1} u_1+  \dots   +d_{mm} u_m =f_m  & \text{ in } \gO \\
  u_1=  \dots  = u_m   =0  & \text {  on } \de \gO 
\end{cases}
$$
 where $F=(f_1, \dots , f_m) \in (L^2 (\gO))^m $,  $U=(U_1, \dots , U_m) $. \par
 \medskip
 
  \begin{definition}\label{sistemilineariaccoppiati} 
 The matrix $D$ or the associated system \eqref{linearsystem} is said to be
 \par
 \smallskip
 \begin{itemize}
 \item  \emph{cooperative} or \emph{weakly coupled} in $ \gO $ if 
 \begin{equation} \label{weaklycoupled} d_{ij}  \leq 0 \; \;   \text{a.e. in }\gO , \quad \text{whenever } i \neq j
 \end{equation}
 \item \emph{fully coupled} in $ \gO $ if it is weakly coupled in $\gO $ and the following condition holds:
 \begin{equation}\label{fullycoupled}
 \begin{split}
  & \forall \, I,J \subset \{ 1, \dots , m \}\, ,\,  I,J \neq \emptyset \, , \, I \cap J = \emptyset \, , \, I \cup J = \{ 1, \dots , m \} \, \\
  &  \exists i_0 \in I \, , \, j_0 \in J \, : \text{meas } (\{ x \in \gO  : d_{i_0j_0} <0 \}) >0
  \end{split}
 \end{equation}
 \end{itemize}
 \end{definition}
 \par
\medskip 
Before going on we fix some notations and definitions. \par
 \begin{itemize}
 \item  Inequalities involving vectors should be understood  to hold componentwise, e.g. if  $\Psi = (\psi _1 , \dots , \psi _m) $,   $\Psi $ nonnegative  means that $\psi _j \geq 0 $ for any index $j=1, \dots , m$. 
 
 \item If $m \geq 2$ and $1 \leq p \leq \infty $  we will consider the Banach  spaces 
   $$
   \Lp =  \left ( L^p(\gO) \right )^m   \; ,   	\;  
    \textbf{W}^{1,p} (\gO)=  \left ( W^{1,p}(\gO) \right )^m
$$
If $p=2$   in particular we have the Hilbert spaces 
$$
  \textbf{L}^2 (\gO) =  \left ( L^2(\gO) \right )^m  \; ,   	\;  
   \textbf{H}^1 (\gO)=  \left ( H^1(\gO) \right )^m
   $$
   and the space 
   $ \Huno =   \left ( H_0^1(\gO) \right )^m $, i.e. the closure in $  \textbf{H}^1 (\gO) $ of the subspace $C_c^{1}(\Omega ; \R^m)$. 
   If $f=(f_1, \dots , f_m)$, $g=(g_1, \dots , g_m)$,   the  related scalar products are  
 \be 
 \begin{split}
 &(f,g)_{\Ldue}= \sum_{i=1}^m (f_i ,g_i )_{L^2(\gO)}=  \sum_{i=1}^m \int _{\gO} f_i \, g_i  \, dx \\ 
 &(f,g)_{\Huno}= \sum_{i=1}^m (f_i ,g_i )_{H_0^1(\gO)}=  \sum_{i=1}^m \int _{\gO} \nabla f_i \,\cdot \, \nabla g_i  \, dx
\end{split}
 \ee

 \item  If $U=(u_1, \dots , u_m) \, , \,  \Psi = (\psi _1, \dots , \psi _m ) \in \Huno$, and the matrix $D$  satisfies \eqref{ipotesiD}
we set  
\be \nabla U \cdot  \nabla \Psi = \sum _{i=1}^m  \nabla u_i \cdot  \nabla \psi _i  
\ee 
\be 
D(x) (U ,\Psi )= \sum _{i,j=1}^m d_{ij}(x) u_i  \psi _j 
\ee
\be \label{BilForm}
\begin{split}
B(U,\Psi)= B^D(U,\Psi)&= \int _{\gO} \left [ \nabla U \cdot \nabla \Psi + D(U, \Psi) \right ] \, dx  \\
&= \int _{\gO} \left [\sum_{i=1}^m  \nabla u_i \cdot \nabla \psi _i +\sum _{i,j=1}^m d_{ij}u_i \psi _j \right ] \, dx
\end{split}
\ee
i.e. 
 $D(x) (U ,\Psi )$ is the action of the bilinear form associated to the matrix $D$ on the pair $(U ,\Psi )$, and   $B$ is
  the bilinear form in $\Huno $ associated to the operator $- \gD + D$.

 \item If $U=(u_1, \dots , u_m) \in \textbf{H}^1 (\gO)$ we say that $U$ weakly satisfies 
 $$
 U \leq  0 \text{ on } \partial \gO  \;  \; ( \; U \geq  0 \text{ on } \partial \gO  \; )
 $$
  if $U^+ \in \Huno $ \  ( $U^- \in \Huno $ ), i.e. if 
 $u_i^+ \in H_0^1 (\Omega )$ ($u_i^- \in H_0^1 (\Omega )$) for any $i=1, \dots , m $.
 
 \item
 If $U=(u_1, \dots , u_m) \in \textbf{H}^1 (\gO)$ and $D$  satisfies \eqref{ipotesiD} we say that $U$ weakly satisfies the inequality
 \be \label{DisugDeboliSistemi} - \gD U + D(x) U \geq 0 \text{ in } \gO
\ee 
if 
 \be \label{DisugDeboliSistemiBis}
 \begin{split}
 \int _{\gO}  \nabla U \cdot  \nabla \Psi   & + D(x) (U ,\Psi )=  \\
 & \int _{\gO}\left [   \sum _{i=1}^m  \nabla u_i \cdot  \nabla \psi _i  +\sum _{i,j=1}^m d_{ij}(x) u_i  \psi _j \right ]  \, dx \geq 0
 \end{split}
\ee
for any nonnegative $\Psi = (\psi _1 , \dots , \psi _m) \in \Huno $ (which is equivalent to require that
$  \int _{\Omega } \left (  \nabla u_i \, \cdot \nabla  \psi   + \sum _{j=1}^m d_{ij}u_j \, \psi   \right )\, dx \geq 0 
$
 for any  $ \psi   \in H_0^1 (\Omega)   $ with $ \psi  \geq 0 $ and any $i=1, \dots , m$).
 \end{itemize}
 \par
 \smallskip
 It is well known that either condition \eqref{weaklycoupled} or conditions \eqref{weaklycoupled} and \eqref{fullycoupled} together are needed in the proofs of maximum principles for systems (see \cite{deF1}, \cite{deFM}, \cite{Si} and the references therein).
In particular if both are fulfilled the strong maximum principle holds as it is stated in the next  theorem (see \cite{DaPaBook}, \cite{deF1}, \cite{deFM}, \cite{Si} for the proof).

\begin{theorem}\label{X-SMP} (Strong Maximum Principle and Hopf's Lemma).  Suppose that \eqref{ipotesiD} and    \eqref{weaklycoupled}  hold  and $U=(u_1, \dots , u_m) \in C^1 (\gO;\R^m)$  is a weak solution of the inequalities 
$$- \gD U + D(x) U \geq 0 \text{ in } \gO \;  \; \text{ and } \; \; U \geq  0 \text{ in }\gO 
$$ 
Then:
\begin{enumerate}
\item for any $k \in \{ 1, \dots, m \} $   either  $u_k \equiv 0 $ or $u_k > 0 $ in $\Omega $; in the latter case if
$u_k \in C^1 (\ov {\gO};\R^m)$, $\Omega $ satisfies the interior sphere condition at 
 $P \in \partial \gO $ and $u_k(P)=0$ then $\frac {\de u_k }{\de \gn }(P) < 0$, where $\gn $ is the unit exterior normal vector at $P$.
\item
if in addition  \eqref{fullycoupled} holds, then the same alternative holds for all $k= 1, \dots , m $, i.e. 
 either $U \equiv 0 $ in $\gO $ or $U>0 $ in $\gO $. In the latter case if $U \in C^1 (\ov {\gO};\R^m)$, $\Omega $ satisfies the interior sphere condition at 
 $P \in \partial \gO $ and $U(P)=0$ then $\frac {\de U }{\de \gn }(P) < 0$, where $\gn $ is the unit exterior normal vector at $P$.
 \end{enumerate}
\end{theorem}

\par
\smallskip
Together with the bilinear form \eqref{BilForm} we consider the  quadratic form 
\begin{equation} \label{formaquadraticalineari}
\begin{split}
Q (\Psi  )= B^D(\Psi, \Psi) &= \int _{\gO} \left (\,  |\nabla \Psi |^2 + D(x) (\Psi ,\Psi ) \, \right ) dx= \\
& \int _{\gO}\left (   \sum _{i=1}^m |\nabla \psi _i |^2 +\sum _{i,j=1}^m  d_{ij}(x)  \psi _i \psi _j \right )  \, dx 
\end{split}
\end{equation}
for  $\Psi  =(\psi _1, \dots , \psi _m) \in   \Huno$. \\
Sometimes we will also write $Q (\Psi ; \gO ) $ instead of $Q (\Psi ) $ specifying the domain. \\
It is easy to see that this quadratic form coincides with the quadratic form  $B^C$
associated to the symmetric linear operator $- \Delta + C $
  where \ \ $C= \frac 12 (D + D^{t}) $ and $D^{t}$ is the transpose of $D$, i.e. 
 \begin{equation}\label{matricesimmetricaassociata}  C=(c_{ij}), \quad     c_{ij} =  \frac 12 \, (d_{ij} + d_{ji})
 \end{equation}
Let us observe that if the matrix $D$ is cooperative, respectively fully coupled, so is  the associate  matrix $D$.
\par
Thus,  let us review some results for a symmetric linear operator  $- \Delta + C $, with $C$ such that 
\be \label{ipotesiC} c_{ij} \in L^{\infty}(\Omega) \quad , \quad c_{ij}=c_{ji} \quad \text{ a.e. in }  \gO
\ee  
Let  us consider the bilinear form
 \be \label{Xbilinearform} B(U,\Phi)= \int _{\gO} \left [ \nabla U \cdot \nabla \Phi + C(U, \Phi) \right ] = \int _{\gO} \left [\sum_{i=1}^m  \nabla u_i \cdot \nabla \phi _i +\sum _{i,j=1}^m c_{ij}u_i \phi _j \right ] 
 \ee

 Using the theory of compact selfadjoint operators we get that
 there exists a sequence $\{\gl _j \}= \{\gl _j (- \Delta + C) \}$ of eigenvalues, with
 $- \infty < \gl _1 \leq \gl _2 \leq \dots $, $\lim _{j \to + \infty} \gl _j = + \infty $, and a corresponding sequence of eigenfunctions $\{ W^j \}$ which weakly solve the systems
 \be \label{sistemaautovalori}
 \begin{cases} - \gD W^j + C  W^j =   \gl _j W^j &  \text{ in }  \gO \\
  W^j =0 & \text{ on } \de \gO
 \end{cases}
 \ee
i.e. if $W^j=(w_1, \dots , w_m)$
$$
\begin{cases}
 - \gD w_1 +  c_{11} w_1+  \dots   +c_{1m} w_m &= \gl _j  w_1 \\
  \dots  \dots & \dots \\ 
 - \gD w_m +  c_{m1} w_1+  \dots   +c_{mm} w_m &=\gl _j  w_m
\end{cases}
$$ 
Moreover by (scalar) elliptic regularity theory applied iteratively to each equation, the eigenfunctions $W^j $ belong at least to $  C^{1}(\gO;\R^m)$ and the eigenvalues can be given a variational formulation. We refer to \cite{DaPaSys}, \cite{DaPaBook} for the construction of the sequences $\gl _j $ and  $\{ W^j \}$ as well as for the  proof of the following theorem, which gives some variational properties of eigenvalues and eigenfunctions.
\par
  \smallskip  
\begin{theorem}\label{Xvarformautov} Let $\Omega $ be a bounded domain in $\R^N$, $N \geq 2 $.  Suppose that $C=(c_{ij})_{i,j=1}^m $ satisfies \eqref{ipotesiC}, and let $\{\gl _j \}$, $\{ W^j \}$ be the sequences of eigenvalues and eigenfunctions satisfying \eqref{sistemaautovalori}.\\
Define the Rayleigh quotient
\be R(V)= \frac {Q(V)}{(V,V)_{\Ldue}} \quad \text{ for } V \in \Huno \; ,  \quad V \neq 0
\ee
with $Q(V)=B(V,V)$ and $B$ as in \eqref{Xbilinearform}.
Then the following properties hold, where  $\textbf{H}_k$ denotes a $k$-dimensional subspace of $\Huno$ and the orthogonality conditions  $V \bot W^k $  or  $V \bot \textbf{H}_k $  stand for the orthogonality in $\Ldue$.  
\begin{itemize}
\item[i)] 
$$\gl _1 =\min  _{\substack{ V\in \Huno \\ V \neq 0 } } R(V) =  \min  _{\substack{ V\in \Huno \\  (V,V)_{\ldue}=1 } } Q(V) 
$$
\item[ii)]    if $k \geq 2$ then 
$$\gl _k = \min  _{\substack{ V\in \Huno \\ V \neq 0 \\ V \bot W^1,\dots , V \bot W^{k-1} }} R(V)
= \min  _{\substack{ V\in \Huno \\  (V,V)_{\Ldue}=1 \\ V \bot W^1,\dots , V \bot W^{k-1} }} Q(V) 
$$  
$$ = \min  _{\textbf{H}_k }  \;  \max  _{ \substack{ V \in \textbf{H}_k \\ V \neq 0  }} R(V)  =\max  _{\textbf{H}_{k-1} }   \;  \min  _{\substack{ V \bot   \textbf{H}_{k-1 }\\ V \neq 0 }} R(V) 
$$
\item[iii)] if $W \in \Huno$, $W \neq 0$, and $R(W)= \gl _1$, then $W$ is an eigenfunction corresponding to $\gl _1$.
\item[iv)] if the system is fully coupled in $\gO $,  then the first eigenfunction does not change sign in $\gO $ and  the first eigenvalue is simple, i.e. up to scalar multiplication there is only one eigenfunction corresponding to the first eigenvalue. 
\item[v)]  \ \  if  $X= \left ( C_c^{1} (\Omega  ) \right ) ^m$ and  we denote by 
$X _ k  
$ 
a $k$-dimensional  subspace of  $X$ then
 $$\gl _k = \inf _{X_k }  \;  \max _{\substack{  v \in X_k \\ v \neq 0 } }  \;R(v) 
$$
\item[vi)] \   let us consider an open subset $\Omega '  \subset \Omega $
and set
$\gl _1 (\gO ') = \gl _1 \left ( - \gD  +  C   \; ; \; \Omega '     \right )$.
  Then
$$
 \lim _{\text{ meas }(\gO ') \to 0} \gl _1 (\gO ') = + \infty 
 $$

\end{itemize}
\end{theorem}

\par
\medskip

 \subsection{Weak Maximum Principle for Cooperative systems} \ \\
 Let us turn back to the (possibly) nonsymmetric  cooperative system \eqref{linearsystem}
 with $D$ satisfying \eqref{ipotesiD} and \eqref{weaklycoupled}.
We consider the associated symmetric  matrix $C$  given by  \eqref{matricesimmetricaassociata}
and we denote by $ \gl _j^{\text{(s)}}=  \gl _j^{\text{(s)}}(-\gD +D; \gO )    $ the eigenvalues of the corresponding symmetric linear operator $- \Delta + C$ and by  $W_j^{\text(s)}$ the corresponding eigenfunctions. \par
The eigenvalues  
$ \gl _j^{\text{(s)}} $ will be called \textbf{symmetric eigenvalues} of the (possibly nonsimmetric) operator $- \Delta + D$.\par
The bilinear form corresponding to the symmetric operator will be denoted by $ B^{\text{s}}(U,\Phi) $, i.e.
 $ B^{\text{s}} $ is as \eqref{Xbilinearform}.
 
  As already remarked, the quadratic form \eqref{formaquadraticalineari} corresponding to the linear operator $- \Delta + D $ coincides with 
   that associated to the  symmetric linear operator $- \Delta + C $.
\par 
\smallskip
\begin{definition} We say that the maximum principle holds for the operator $- \gD + D$ in an open set  $\gO ' \subseteq \gO $ if  for any $U \in \textbf{H}^1 (\gO ')$ such that 
  $U \leq  0 $ on $ \de \gO ' $ (i.e. $U^{+} \in \textbf{H}_0^1 (\gO ')$ )  and
   $- \gD U + D(x)U \leq 0 $ in $\gO ' $ (i.e. 
$\int \nabla U \cdot  \nabla \Phi + D(x) (U, \Phi ) \leq 0 $
for any nonnnegative  $\Phi \in \textbf{H}_0^1 (\gO ') $ )
 it holds  that  $U \leq 0 $ a.e. in $\gO $.
\end{definition}
Let us denote by $  \gl _j^{\text{(s)}}  (\gO ')   $, $j \in \N^+$, the sequence of the symmetric eigenvalues of the   linear operator $- \Delta + D $ (i.e. the eigenvalues of $- \Delta + C $) in an open set $\gO ' \subseteq \gO$.
\par 
\smallskip
\begin{theorem}\label{WMPSystems} [Sufficient condition for weak maximum principle] \  Under the hypothesis  \eqref{ipotesiD} and \eqref{weaklycoupled}, if
 $  \gl _1^{\text{(s)}}  (\gO ') >0  $ then the maximum principle holds for the operator $- \gD + D$ in $\gO ' \subseteq \gO $. \\
\end{theorem}
\begin{proof} By the variational characterization of  the first eigenvalue given in Theorem \ref{Xvarformautov} we have 
$$
\gl _1 ^{\text{(s)}} = \text{ min } _{ \substack {V\in \textbf{H}^1 (\gO ') \\ V \neq 0}} R(V) > 0 
$$
 so that $Q(V)=B^{\text{s}}(V,V) >0 $ for any $V \neq 0 $ in $ \textbf{H}_0^1 (\gO ') $. \\
Assume that \  $U \leq  0 $ on $ \de \gO ' $ \ and  \  \ $- \gD U + D(x)U \leq 0 $ in $\gO ' $. 
Then, testing the equation with $U^{+}=(u_1^{+},\dots , u_m^{+})$, writing in the $i$-th equation 
$u_j= u_j^{+} - u_j^{-}$ for $i \neq j $, and recalling that $-c_{ij}u_i^{+} u_j^{-} \geq 0 $ if $i\neq j$, we obtain that 
$B^{\text{s}}(U^{+},U^{+}) \leq 0 $, which implies $U^{+} \equiv 0 $ in $ \gO ' $.
\end{proof}
 
 As an almost immediate consequence we get a quick proof of the  following ''classical'' and ''small measure'' forms of the weak maximum principle (see   \cite{BuSi2}, \cite{deFM},  \cite{PrWe}, \cite {Si}).

\begin{theorem} \label{WeakMaximumSistemi}  Assume that  \eqref{ipotesiD} and \eqref{weaklycoupled} hold. 
\begin{itemize}
\item[i) ] If  $D$ is a.e. nonnegative definite in $\gO ' $ then the maximum principle holds for
$- \gD + D$ in $ \gO ' $. 
\item[ii) ]
There exists $\gd >0 $, depending on $D$, such that for any subdomain $\gO ' \subseteq \gO $ the maximum principle holds for $- \gD + D$ in $\gO ' \subseteq \gO $ provided $|\gO ' | \leq \gd$. 
\end{itemize}
\end{theorem}
\begin{proof} i) \  If the matrix $D$ is  nonnegative definite then 
$$
Q(\Psi)=B^{\text{s}}(\Psi , \Psi ) \geq \int _{\gO } |\nabla \Psi |^2 >0 
\text{ for any  } \Psi \in  \textbf{H}_0^1 (\gO ')  \setminus \{ 0 \} \; .
$$
Hence $  \gl _1^{\text{(s)}}  (\gO ') >0  $,  and by Theorem \ref{WMPSystems} we get i).\\   
ii) It is a consequence of Theorem \ref{WMPSystems} and   Theorem \ref{Xvarformautov},  vi).
\end{proof}
 
 \begin{remark}
Obviously the converse of Theorem \ref{WMPSystems} holds if $D=C$ i.e. if $D$ is symmetric: if the maximum principle holds for $- \gD + C$ in $\gO '$ then $\gl _1^{\text{(s)}} (\gO ' )>0$.
Indeed if $\gl _1^{\text{(s)}} (\gO ' )\leq 0 $, since the system is cooperative (and symmetric), there exists a corresponding nontrivial nonnegative first eigenfunction $\Phi _1 \geq 0 $, 
$\Phi \not \equiv 0$, and  the maximum principle does not hold, since $- \gD \Phi _1 + C\, \Phi _1 = \gl _1 \Phi _1 \leq 0 $ in $\gO '$, $\Phi _1 =0 $ on $\de \gO '$, 
while  $\Phi _1 \geq 0 $ and
$\Phi _1 \neq 0$.
However the converse of Theorem \ref{WMPSystems} is not true for general nonsymmetric systems, since there is an equivalence between the validity of the maximum principle for the operator $- \gD + D$ and the positivity of another eigenvalue, the  \emph{principal eigenvalue} $\tilde {\gl _1}$,  (we recall below the definition given in \cite{BeNiVa}), and  the inequality 
$\tilde {\gl _1} (\gO ') \geq    \gl _1^{\text{(s)}}  (\gO ' )  
$,  which can be  strict, holds.
\end{remark}
 \par  
\medskip

\begin{definition} \label{DefinizPrincipaleigenvalue}

The \textbf{principal eigenvalue}  
of the operator  
 $- \gD + D$ in an open set $\gO ' \subseteq \gO $ is defined as  
\be
\begin{split}
 \tilde {\gl _1} (\gO  ') &= 
   \sup \{ \gl \in \R : \exists \, \Psi \in W^{2,N}_{loc} (\gO ' ; \R^m) \; \text { s.t. } \\ 
&      \Psi  >0  \; , \;  - \gD \Psi + D(x) \Psi - \gl \Psi \geq 0 \text{ in } \gO '\}
\end{split}
\ee
\end{definition}
\par
\medskip

Let us recall some of the properties of the principal eigenvalue.
We refer to \cite{BuSi2}  for the proofs of items i) --  iii), as well as for  references on the subject, and to  \cite{DaPaSys}, \cite{DaPaBook} for the proof of iv). \par

\begin{theorem}\label{principaleigenvalue} Assume that the matrix $D$  is fully coupled in an open set  $\gO ' \subseteq \gO $, i.e. \eqref{fullycoupled} holds. Then:
\begin{itemize}
\item[i) ] there exists a positive eigenfunction $\Psi _1 \in  W^{2,N}_{loc} (\gO ' ; \R^m)$ which satisfies  
\be \label{autofprincipale}   
\begin{cases}
- & \gD \Psi _1 + D(x)  \Psi _1  = \tilde {\gl _1 } (\gO ') \Psi _1 \text { in } \gO ' \\
 & \Psi _1  >0  \;   \text { in }  \gO '  \\
 & \Psi _1  =0   \;  \text { on } \de \gO ' 
\end{cases}
\ee
Moreover the principal eigenvalue is simple, i.e. any function that satisfy \eqref{autofprincipale} must be a multiple of $\Psi _1$
\item[ii) ]  the maximum principle holds for the operator $- \gD + D$ in $\gO '$ if and only if $\tilde {\gl _1} (\gO ' ) >0$
\item[iii) ] if there exists  $\Psi \in  W^{2,N}_{loc} (\gO ' ; \R^m) $ such that   $\Psi  >0$ and   $ - \gD \Psi + D(x) \Psi  \geq 0 $ in $\gO '$, then either 
$\tilde {\gl _1} (\gO  ') >0$ or $\tilde {\gl _1} (\gO ') =0 $ and $\Psi = c\, \Psi _1$ for some  $c>0$
\item[iv) ]  $\tilde {\gl _1} (\gO ' ) \geq  \gl _1^{\text{(s)}}  (\gO ' ) $, with equality if and only if $\Psi _1 $ is also the first eigenfunction of the symmetric operator $- \gD + C$ in $\gO '$ ,  $ C=\frac 12 (D + D^{\text{t}})$.
If  this is the case the equality $C(x) \Psi _1 =D(x) \Psi _1$ holds and,  if $m=2$, this implies that $d_{12}= d_{21}$.
\end{itemize}
\end{theorem}

 \subsection{Comparison principles for semilinear elliptic systems} \ \\
 Let us  consider a semilinear elliptic system 
 of the type
 
\begin{equation}  \label{X-semilinear system}
\begin{cases} - \gD U =F(x,U) \quad &\text{in } \gO \\
U= 0 \quad & \text{on } \de  \gO
\end{cases}
\end{equation}
i.e.
$$
\begin{cases}  
 - \gD u_1  = f_1(x,u_1,\dots , u_m) & \text{ in } \gO  \\
  \dots  \dots & \dots \\ 
 - \gD u_m = f_m(x,u_1,\dots , u_m)   & \text{ in } \gO \\
 u_1=  \dots = u_m  =0 & \text { on } \de \gO 
\end{cases}
$$
for the unknown vector valued function $U=(u_1, \dots ,u_m): \Omega  \to \R^m$,    where $\gO $ is a bounded domain in $\R^N $ and  $F=(f_1, \dots , f_m ): \ov {\gO}\times \R^m \to \R^m $  is a $C^{1 }$ function.\par
\smallskip

 A weak solution of \eqref{X-semilinear system} is a function 
 $U \in \Huno $ such that the function 
 $x \mapsto F(x,U(x))$ belongs to $\textbf{L}^q (\Omega )$, with $q>1 $ if $N=2$, $q= \frac {2N} {N+2} $ if $N \geq 3$  (note that $\frac {2N} {N+2}$ is the conjugate exponent of the critical Sobolev exponent $2^*= \frac {2N}{N-2}$) and 
 \be \label{X-WeakSolution}
 \int _{\Omega } \nabla U \, \cdot \, \nabla  \Phi \, dx = \int _{\Omega } F(x,U) \,  \cdot \,  \Phi \, dx  \quad \forall \; \Phi \in \Huno
 \ee
  If $U,V \in \textbf{H}^1(\Omega )$ we write $U \leq V $ on $\partial \Omega $, if the difference $U-V$ weakly satisfies the inequality $U- V \leq 0 $ on $\partial \Omega $, i.e.  if $(U-V)^+  \in \Huno $. \par
 Moreover we say that $U$ satisfies in a weak sense the inequality 
 \be \label{X-SemilinearEllIneqn}
- \Delta U \geq \; \; (\, \leq \,) F(x,U) \quad \text{in } \Omega
\ee
if for any $i=1, \dots ,m $ 
the component $u_i$ of $U$ weakly satisfies   
$$
- \Delta u_i \geq \; \; (\, \leq \,) f_i(x,U) \quad \text{in } \Omega \; ,  \; \;  \text{ i.e. } 
$$ 
 \be \label{X-WeakIneq}
 \int _{\Omega } \nabla u_i \, \cdot \nabla  \gf \, dx \geq \; \; (\, \leq \,) \quad  \int _{\Omega } f_i(x,U) \, \gf \, dx  
 \ee
 for any $\gf \in H_0^1 (\Omega )$ with $\gf \geq 0 $ in $\Omega $.
 This is equivalent to require that
 \be \label{XX-WeakIneq}
 \int _{\Omega } \nabla U \, \cdot \nabla  \Phi \, dx \geq \; \; (\, \leq \,) \quad  \int _{\Omega } F(x,U) \, \cdot  \Phi \, dx  
 \ee
 for any $\Phi \in \Huno $ with $\Phi \geq 0 $ in $\Omega $.
 \par
 \medskip

\begin{definition} \label{X-semilinearicoupled} We say that the system \eqref{X-semilinear system} is
\par
\smallskip
\begin{itemize}
\item  \textbf{cooperative or weakly coupled} in an open set $\gO ' \subseteq \gO $ if 
\be \label{X-EqSemilCoupled}  \frac {\de f _i} {\de s _j} (x, s_1, \dots , s_m) \geq 0 \quad \text{ for every } (x,s_1,\dots ,s_m ) \in  \gO ' \times \R^m
\ee
and every  $i,j=1,\dots ,m$ with $i \neq j$.
 
 \item    \textbf{fully  coupled} in an open set  $\gO ' \subseteq \gO $ along  $U \in \Huno \cap C^0(\Omega ; \R^m)$   if it is cooperative in $\gO '$ and 
    in addition $\forall I,J \subset \{1, \dots ,m \}$ such that $I \neq \emptyset $, $J \neq \emptyset $, $I \cap J = \emptyset $, $I \cup J = \{1, \dots ,m \} $ there exist $i_0 \in I$, $j_0 \in J $ such that 
\be \label{X-EqSemilFullyCoupled} \text{meas }(\{ x \in \gO ' :   \frac {\de f _{i_0}} {\de s _{j_0}} (x, U(x)) > 0 \}) >0
\ee
 \end{itemize}
\end{definition}
\par
\medskip

As a consequence of Theorem \ref{WeakMaximumSistemi} and Theorem  \ref{X-SMP} the following comparison principles hold  (see \cite{DaPaBook} for the proof).

\begin{theorem}[Weak comparison principle in small domains for systems] \label{X-confrontodebole}
Let  $\Omega  $ be a domain in $\R^N $,  $F: \Omega  \times \R ^m  \to \R ^m $   a $C^1$ function and assume that  \eqref{X-EqSemilCoupled} holds. Let    
$A >0 $ and 
$U,V \in \textbf{H}^1 (\gO ) \cap L^{\infty}(\Omega )$  such that 
$$
 \Vert U \Vert  _{ L^{\infty} (\Omega ) }   \leq A  \; , \;  \Vert   V \Vert _{ L^{\infty} (\Omega ) }   \leq A 
 $$
 Then there exists $\delta >0 $, depending on $F$ and $A$   such that the following holds:\par 
\noi if $\Omega ' \subseteq \Omega  $ is a bounded subdomain of $\Omega $,  \ $\text{meas}_N \; ([u>v] \cap \Omega ' ) < \delta $ and 
\begin{equation} \label{X-comparison} 
\begin{cases} - \gD U \leq F(x,U) \;, \;  - \gD V \geq F(x,V) \quad &\text{ in } \gO '\\
U \leq V \quad & \text{on }  \partial \Omega '    
\end{cases}
\end{equation}
  then $U \leq V $ in $\Omega ' $.
 \end{theorem} 

\begin{theorem}[Strong Comparison Principle for   systems] \label{X-Strong Comparison Principle}
Let $\Omega $  be a (bounded or unbounded) domain in $\R^N$, and let  $U,V  \in   C^1 ( \Omega ) $ weakly   satisfy  
\begin{equation} \label{X-SCP} 
\begin{cases}  - \gD U   \leq F(x,U) \quad ; \quad  - \gD v   \geq F(x,V) \quad &\text{in } \gO \\
U \leq V  \quad & \text{in } \gO \\
\end{cases}
\end{equation}
where $F(x,U): \Omega \times \R^m \to \R^m $ is a $C^1 $ function and \eqref{X-EqSemilCoupled} holds.  
\begin{enumerate}
\item For every $i \in \{ 1, \dots ,m \}$ the following holds:   either $ u_i \equiv v_i $ in $\Omega $ or  $u_i< v_i $ in $\Omega $ and, in the latter case, if
$u_i, v_i  \in   C^1 (\ov {  \Omega } ) $, 
   $u_i(x_0)=v_i(x_0 )$ at a point  $x_0 \in \partial \Omega $  where  the interior sphere condition is satisfied 
then $  \frac{\de u_i}{\de s} (x_0)< \frac{\de v_i}{\de s} (x_0)  $  for any inward
directional derivative.
\item If moreover $U \in C^1 (\Omega ; \R^m )$ is a solution of   \eqref{X-semilinear system} and the system is fully coupled along $U$ in $\Omega $ (i.e. also \eqref{X-EqSemilFullyCoupled} with $\Omega ' = \Omega$ holds)
then either $U \equiv V $ in $\Omega$ or $U<V $ in $\Omega $ (i.e. the same alternative holds for any component $u_i$).
In the latter case assume that $U,V  \in   C^1 ( \ov { \Omega } ) $ and
  let $x_0 \in \partial \Omega $ a point where    $U(x_0)=V(x_0 )$ and the interior sphere condition is satisfied.
Then $  \frac{\de U}{\de s} (x_0)< \frac{\de V}{\de s} (x_0)  $  for any inward
directional derivative.
\end{enumerate}
\end{theorem}

  \subsection{Morse index of a solution}
  \begin{definition} \label{MorseIndex} \ 
\begin{itemize}
\item [i)]
Let $U \in \Huno \cap \Linf$ be a  weak solution of \eqref{genprobsys}. We say that $U$ is linearized stable (or  has zero Morse index) if the quadratic form 
\begin{equation} \label{Xformaquadraticalinearizzato} 
\begin{split}
Q_U (\Psi ; \gO ) &= \int _{\gO} \left [ |\nabla \Psi |^2 - J_F (x, U(x))(\Psi ,\Psi )\right ] dx= \\
& \int _{\gO}\left [   \sum _{i=1}^m |\nabla \psi _i |^2 -\sum _{i,j=1}^m \frac {\de f_i}{\de s_j}(x, U(x)) \psi _i \psi _j \right ]  \, dx \geq 0
\end{split}
\end{equation}
for any $\Psi =(\psi _1, \dots , \psi _m) \in C_c^1 (\gO;\R^m)$ where $J_F (x,U(x))$ is the jacobian matrix of $F(x,S)$ with respect to the variables $S=(s_1, \dots , s_m)$ computed at $S=U(x)$.
\item [ii)] $U$ has (linearized) Morse index  equal to the integer $m=m (U)\geq 1$ if $m $ is the maximal dimension of a subspace of  $ C_c^1 (\gO;\R^m)$ where the quadratic form is negative definite.
\item [iii)] $U$ has infinite (linearized) Morse index  if for any integer $k$ there exists  a $k$-dimensional subspace of  $ C_c^1 (\gO;\R^m)$ where the quadratic form is negative definite.
\end{itemize}
\end{definition}

The crucial, simple remark that  allowed  to extend some of the symmetry results known for equations to the case of systems in \cite{DaPaSys} and \cite{DaGlPa1}, is that  the quadratic form associated   
to the linearized operator at a solution $U$, i.e. to the linear  operator 
\be \label{linearizedoperator} L_U (V) = - \gD V - J_F (x, U) V 
\ee 
 which in general is not selfadjoint, coincides with the quadratic form corresponding to  the  selfadjoint operator 
 \be \label{symmetriclinearizedoperator}L^s_U (V) = - \gD V - \frac 12 \left (J_F (x, U) + J_F ^{\text{t}} (x, U)  \right )  V 
\ee 
where $J_F ^{\text{t}}$ is the transpose of the matrix $J_F $. \\
   Therefore the \emph{symmetric eigenvalues} of $L$,  i.e. the eigenvalues of $L^s_U$, as defined in Section 2.2 can be exploited to study the symmetry of the solution $U$, using the information on its Morse index. \par
 As in section 2.2 we denote by  $\gl _k ^{\text{s}} = \gl _k (- \Delta  - \frac 12 \left (J_F (x, U) + J_F ^{\text{t}} (x, U)  \right ) ; \Omega )$ and $W^k$,  $k \in \N^+$,  the \emph{symmetric} eigenvalues and eigenfunctions of $L_U  = - \gD V - J_F (x, U)  $ in an open set $\Omega $.
     Then we have 
\par
\smallskip

\begin{proposition}\label{X-CarattIndiceMorse} Let $\gO $ be a bounded domain in $\R^N$. Then the Morse index of a solution $U$ to \eqref{X-semilinear system} equals the  number of negative \emph{symmetric} eigenvalues of the linearized operator $L_U$.
\end{proposition}

\begin{proof} 
 Let us denote by $\mu (U)$ the number of negative symmetric eigenvalues of $L_U$.
 If the quadratic form $Q_U$ defined in \eqref{Xformaquadraticalinearizzato}  is negative definite on a $k$-dimensional supspace of $C_c^1(\Omega\cup\Gamma)$, then, by \textit{iii}) of Theorem \ref{Xvarformautov}   we have that the $k$-th eigenvalue $ \gl _k (- \Delta +C ; \Omega )$   is negative. Hence $\mu(u)\ge m(u)$.
On the other hand if there are $k$ negative symmetric eigenvalues,   by \textit {v}) of Theorem \ref{Xvarformautov}  there is a $k$-dimensional supspace of $C_c^1(\Omega\cup\Gamma)$ where the quadratic form $Q_u$ is negative definite, hence  $ m(u) \geq  \mu(u)$.
   \end{proof}
 \par

 \section{Proof of the symmetry results}
  \subsection{On the $k$-sectional foliated Schwarz symmetry} \ \\
  From now on  we will consider the case of system \eqref{modprob}  in a  bounded  $k$-rotationally symmetric domain.
   Let us fix  some notations.  \\
   For a unit vector $e \in S^{N-1}$ we consider the hyperplane 
$$ H(e)= \{x \in \R^N\::\: x \cdot e= 0\}
$$
orthogonal to the direction $e$ 
 and the open half domain 
 $$\Omega (e)= \{x \in \Omega \::\: x \cdot e>0\}
 $$ 
 We then set  
 $$ \sigma_e(x) =x -2(x \cdot e)e \; , \;  x \in \Omega \, , 
 $$ 
i.e.   $\sigma_e: \Omega  \to \Omega $ is the \emph{reflection with respect to the hyperplane}
$H(e)$. 
Finally if $U: \Omega \to \R ^m $ is a continuous function we define the \emph{reflected function}
$U^{\gs (e)}: \Omega \to \R ^m $ defined by 
$$  U^{\gs (e)} (x) = U(\gs_e(x))
$$

\par
\smallskip
We will use in the sequel  the   rotating plane method, a variant of the moving plane method in the version of Berestycki and Nirenberg (see \cite{BeNi}),    exploited e.g. in   \cite{PaWe} and subsequent papers on foliated Schwarz symmetry of solutions of equations. 

 \begin{theorem}[Rotating Planes method for systems ] \label{Rotating Planes Systems}
 Let $\Omega $ be a bounded  $k$-rotationally symmetric domain in $\R^N$,  $F \in C^1 (\, [0, \infty ) \times \R^{N-k} \times \R ^m \, ; \, \R ^m \, ) $ and 
 $U \in \Huno \cap C^0(\overline{\Omega}; \R ^m )$  a weak solution of \eqref{modprob}. Assume that the system \eqref{modprob} is fully coupled along $U$ in $\gO $ and there exists a direction 
 $  e_{\gth _0}= (\cos (\gth _0 ), \sin (\gth _0),0, \dots ,0 ) $   such that 
$$
U< U^{\gs ( e_{\gth _0} )}  \; \text{ in }  \; \gO ( e_{\gth _0} )
$$
 Then there exists a direction  $ e_{\gth _1}= (\cos (\gth _1 ), \sin (\gth _1),0, \dots ,0 )$, with $\gth _1 > \gth _0$, such that 
$$
U \equiv U^{\gs (e_{\gth _1})}  \quad \text{ in }   \;  \Omega (e_{\gth _1}) 
$$
and 
$$  U < U^{\gs (e_{\gth })} \quad  \text{ in } \; \Omega (e_{\gth })\; \; \forall \;    \gth \in (\gth _0, \gth _1)
$$
\end{theorem}

\begin{proof} Let us observe that the functions $U^{\gs (e_{\gth })}$ satisfy the same equation as $U$, namely
$- \Delta U^{\gs (e_{\gth })} = F(|x'|,x'', U^{\gs (e_{\gth })})$ in $\Omega $, and both 
 $\Vert U   \Vert _{ L^{\infty} (\Omega (e_{\gth }))  } $ and  $  \Vert U ^{\gs (e_{\gth })}   \Vert _{ L^{\infty} (\Omega (e_{\gth }))  } $ are bounded by 
 $    \Vert U   \Vert _{L^{\infty}(\Omega )}=: A
 $.\\
 Let  us fix $\gd = \gd (A) $ as in Theorem \ref{X-confrontodebole} and observe that $\gd $ is independent of $\gth $, and the functions  $U, U ^{\gs (e_{\gth })}$ satisfy 
\be  
\begin{cases}
- \gD U = F(|x'|,x'',U)  \quad ; \quad   - \gD U^{\gs (e_{\gth })}   =  F(|x' |,x'',U^{\gs (e_{\gth })}) \; \;    &\text{ in } \;  \Omega (e_{\gth }) \\
U = U^{\gs (e_{\gth })} \quad &\text{on }  \partial  \Omega (e_{\gth })
\end{cases}
\ee
Let us set 
$
\gTh = \{ \gth \geq  \gth _0 : U<  U^{\gs (e_{\gth ' })} \text{ in } \gO ( e_{\gth  '} ) \; \forall \gth ' \in (\gth _0, \gth)\} 
$ and let us  show  that
the set $\gTh $ is nonempty and contains an interval $[\gth _0 , \gth _0 + \gep )$ for $\gep >0 $ sufficiently small.
Indeed we can take a compact set $K \subset \Omega (e_{\gth _0}) $ such that  \ $|\Omega (e_{\gth _0})  \setminus K | \leq \frac {\gd}2$   \ and   
$m= \min _K (U^{\gs ( e_{\gth _0} )} -U)>0 $. By continuity if $\gth $ is close to $ \gth _0 $ we have that
$K \subset \Omega (e_{\gth }) $, 
$(U^{\gs ( e_{\gth } )} -U ) \geq \frac m2 >0 $ in $K$,   $|\Omega (e_{\gth })  \setminus K | \leq  \gd$ and
$(U^{\gs ( e_{\gth } )} -U) \geq 0 $ on $ \partial ( \,  \Omega (e_{\gth }) \setminus K \, )$. 
Then  by the weak comparison principle in small domains (Theorem  \ref{X-confrontodebole}) we get that $U \leq U^{\gs ( e_{\gth } )} $ in $ \Omega ' = \Omega ( e_{\gth } ) \setminus K $ and hence in  $\gO ( e_{\gth } )$. Moreover $U < U^{\gs ( e_{\gth } )} $ in $  \Omega ( e_{\gth } )  $ by the strong comparison principle (Theorem \ref{X-Strong Comparison Principle}). 
So the set $\gTh $ is nonempty, and  is bounded from above by $\gth _0 + \gp $, since, considering the opposite direction, the inequality between $U$ and the reflected function gets reversed. Let us set 
  $\gth _1= \sup \gTh $. \\
  We claim that $U \equiv  U^{\gs ( e _{\gth _1}) } $ in $\gO ( e_{\gth _1} )$. 
  Indeed, if this is not the case, we get  $U < U^{\gs ( e _{\gth _1}) }$ in $\gO ( e_{\gth _1})$ by the strong comparison principle (Theorem \ref{X-Strong Comparison Principle}), since by continuity $U \leq  U^{\gs ( e _{\gth _1}) } $ in $\gO ( e_{\gth _1} )$. Then, using again the weak comparison principle in small domains and the previous technique we  get $U<  U^{\gs ( e _{\gth } )}$ in $\gO ( e_{\gth } )$ for $\gth > \gth _1 $ and close to $\gth _1$, contradicting the definition of $\gth _1$.
\end{proof}

\par
\medskip

  Let us define, with a little abuse of notations,
        \be \label{Definizione di S / k-1} S^{k-1}= \{  e \in S^{N-1}:  e \cdot  e_j =0 \,, \,  j= k+1, \dots , N \}
        \ee
\par

        A sufficient condition for the $k$-sectional foliated Schwarz symmetry is the following.
        \par
        \smallskip
        \begin{proposition}
\label{SuffCondSectFSS1Sistemi} Let $\Omega $ be a   $k$-rotationally symmetric domain   in $\R^N$,   $2 \leq k \leq N$, and 
 $U \in \Huno \cap C^0(\overline{\Omega})$ a weak solution of \eqref{modprob} where $F=F(r,x'',S) \in C^1 ([0, \infty ) \times \Omega '' \times \R ^m ; \R ^m) $. Assume that the system is fully coupled along $U$ in $\gO $ and  that  $ \; \forall \; e \in S^{k-1}$ 
  \be \label{AlternDisugSistemi} \text{ either } U \geq U^{\gs  (e) } \; \text { or } \;  U \leq U^{\gs  (e) } \; \text{ in }  \Omega (e)
 \ee
 Then $U$ is $k$-sectionally foliated Schwarz symmetric.
 \end{proposition}
 \par
 \smallskip
The proof is similar to the one given, for the case $k=N$, in \cite{DaPaSys}, with some obvious change.
      \par
 \bigskip
 Let us consider a pair of orthogonal directions $\eta _1, \eta _2 \in S^{k-1}$,  the polar coordinates $(\gr , \gth )$ in the plane spanned by them and the corresponding cylindrical coordinates $(\gr, \gth, \tilde {y})$, with 
 $ \tilde {y} \in \R^{N-2}$.
  Then we 
  define for $U \in C^2(\ov {\Omega}; \R^m)$ the angular derivative 
 \be U_{\gth} = U_{\gth (\eta_1, \eta_2)}
 \ee  
  (trivially extended if $\gr =0 $)   which solves the linearized system
 \be \label{SistemaLinearizzato} 
 \begin{cases}
 - \gD U _{\gth} - J_F (|x|, U) U_{\gth} &=0 \quad \text{ in } \gO  \\
 U _{\gth} =0  \quad \text{ on }  \partial \gO
 \end{cases}
\ee 
and, if $ e \in \text{span }(\eta_1, \eta_2)$ and $U \equiv U^{\gs (e)}$ in $\Omega (e)$, also the system
 \be \label{SistemaLinearizzatoCappa} 
 \begin{cases}
 - \gD U _{\gth} - J_F (|x|, U) U_{\gth} &=0 \quad \text{ in } \gO (e)  \\
 U _{\gth} =0  \quad \text{ on }  \partial \gO (e)
 \end{cases}
\ee 
 \par
Using the properties of the principal eigenvalue and of the corresponding  eigenfunction we deduce, as in the case $k=N$, (see   \cite{DaPaSys}, \cite{DaPaBook},  \cite{DaGlPa1}),  the following sufficient conditions for the
$k$-sectionally foliated Schwarz symmetry.
\par
\smallskip

\begin{theorem}[Sufficient conditions for sectional FSS-Sistems] \label{SuffCondSectFSS2Sistemi} 
Let $\Omega $ be a   $k$-rotationally symmetric domain   in $\R^N$,   $2 \leq k \leq N$, and   $U \in C^2(\ov {\Omega}; \R^m)$ a solution  of \eqref{modprob}, where
$ F \in C^1([0, R ] \times \R ^m ; \R^m )$. Then $U$ is $k$-sectionally foliated Schwarz symmetric provided one of the following conditions holds:
\begin{itemize}
\item[i) ] there exists a direction $e \in S^{k-1}$ such that $U \equiv U^{\gs (e)}$ in $\Omega (e)$  and the principal eigenvalue $\tilde {\gl } _1 (\gO (e)) $ of the linearized operator $L_U  = - \gD  - J_F (x, U)  $ in $\gO (e)$ is nonnegative.
\item[ii) ] there exists a direction $e \in S^{k-1}$ such that either $U < U^{\gs (e)} $ or $U > U^{\gs (e)} $ in $\Omega (e)$
\end{itemize}
\end{theorem}

\begin{remark} \label{InterplaySymmPrinc}
Let us observe that in Theorem \ref{SuffCondSectFSS2Sistemi}  it is the nonnegativity of the \emph{principal} eigenvalue the crucial hypothesis, while the information we  get in the  sequel will concern  the \emph{symmetric} eigenvalues of the linearized system. Therefore in the proofs that follow there will be an  interplay and a comparison between the principal eigenvalue and the first symmetric eigenvalue in the cap $\Omega (e)$.
\end{remark}

 If  $U$ is a solution of \eqref{modprob}, $e \in S^{k-1}$ and the system is fully coupled along $U$ in $\gO $,  then the difference $W=W^{e}=U-U^{\gs (e)} = (w_1, \dots , w_m)$ satisfies a linear system in $\Omega $, which is fully coupled in $\Omega $ and $\Omega (e)$:
 \par

  \begin{lemma}\label{EqDifferSistemi}  \ The following assertions hold.
  \begin{itemize}
\item[i)  ]   Assume that $U \in C^1 (\ov{ \Omega } ; \R^m )$ is a  solution of \eqref{modprob} and that the system is fully coupled along $U$ in $\gO $.
  Let us define  for any direction $e\in S^{k-1}$ the matrix $\; B^{e}(x) =\left (b_{ij}^{e}(x) \right )_{i,j=1}^m $, where
\be \label{CoeffEquazDifferenzaSistemi}
 b_{ij}^{e}(x) = - \int _0^1 \frac {\de f _i} {\de s_j} \left ( \,|x|, tU(x)+ (1-t)U^{\gs (e)}(x)  \,  \right ) \, dt  
\ee 
   Then  for any   $e\in S^{k-1}$   the function   $W^{e}=U-U^{\gs (e)} $  satisfies  in  $\gO (e)$   the linear system 
\be \label{EquazDifferenzaSistemi}
\begin{cases}  - \gD W^{e} + B^{e}(x) W^{e} & =0 \quad \text{ in }  \gO (e) \\
W^{e}& =0 \quad \text{ on }  \de \gO (e) 
\end{cases}
\ee 
 which  is fully coupled in   $\gO (e)$.
 \par
 \smallskip
 \item[ii) ]  If $\Psi = (\psi _1, \dots , \psi _m) \in \textbf{H}_0^1 (\Omega (e))$   let $Q^e (\Psi ; \Omega (e) )$ denote the quadratic form associated to the system \eqref{EquazDifferenzaSistemi} in
  $\Omega (e)$,   i.e.
 \be \label{QuadraticFormDifferenzaSistemiCappe}
 \begin{split}
 Q^e (\Psi ; \Omega (e) )&= \int _{\Omega (e)}  \left ( \, |\nabla \Psi |^2 + B^e (\Psi, \Psi ) \, \right ) \, dx \\
 &= \int _{\Omega (e)}  \left ( \, \sum _{i=1}^m |\nabla \psi _i|^2 + \sum _{i,j=1}^m b_{ij}^e \, \psi _i \, \psi _j  \, \right ) \, dx 
 \end{split}
 \ee
Then 
\be \label{UgPartiPosNeg-Q-e}  Q^{e} (\,  W^{e}; \gO (e)  \,) = \int _{\gO (e)} \left [ \,  |\nabla (W^{e}) |^2 + B^{e}(\,  W^{e} ,W^{e}  \,) \, \right ] \, dx = 0   
\ee
 while for the positive and negative parts of $W^{e}$ the following holds:
\be \label{DisugPartiPosNeg-Q-e}  Q^{e} (\,  (W^{e})^{\pm}; \gO (e)  \,) = \int _{\gO (e)} \left [ \,  |\nabla (W^{e})^{\pm} |^2 + B^{e}(\,  (W^{e})^{\pm} ,(W^{e})^{\pm}  \,) \, \right ] \, dx \leq 0   
\ee
 \end{itemize}
\end{lemma}

\begin{proof} 
From the equation $- \gD U = F (|x|,U(x)) $ we deduce that the reflected function $U^{\gs (e)} $ satisfies the equation
$- \gD U^{\gs (e)}  = F (|x|,U^{\gs (e)}(x) ) $ and hence  the difference $W^{e}=U-U^{\gs (e)}= (w_1, \dots , w_m) $  satisfies  
 $$- \gD W^{e}  = F(|x|, U)-F(|x|, U^{\gs (e)})
 $$
 Let us set $V= U^{\gs (e)}$.
For any $i=1, \dots ,m $ we have that  
 \be \nonumber
  \begin{split}
 & f_i(|x|,U(x))- f_i(|x|,V(x)) =   \\
 \sum _{j=1}^m \int _0^1 & \frac {\de f _i} {\de s_j}   \left ( |x|, tU(x)+ (1-t)V(x)   \right ) \, (u_j(x)-v_j(x)) \, dt     
 \end{split}
\ee 
 
 As a consequence
     $W^{e}    $  satisfies    \eqref{EquazDifferenzaSistemi}. Moreover if $i \neq j $ then  $b_{ij}^{e}(x) \leq 0$ by  \eqref{X-EqSemilCoupled} ,  so that the linear system \eqref{EquazDifferenzaSistemi}  is weakly coupled. \par
If $U \in C^1 (\ov{ \Omega } ; \R^m )$ is a solution of   \eqref{modprob} and the system is fully coupled along $U$
then the linear  system   associated to the matrix $ B^{e}$  is fully coupled   in $\Omega $.
Indeed if $i_0 \neq j_0 $ and $\frac {\de f _{i_0}} {\de s _{j_0}} (x, U(x)) > 0$ then, since $ \frac {\de f _i(y)} {\de s_j} \geq 0$ for every $y \in \Omega $, we get that \par
$b_{ij}(x) = - \int _0^1 \frac {\de f _i} {\de s_j} \left [|x|, tU(x)+ (1-t)V(x)  \right ] \, dt    <0 $. \\
Since $B^{e}$ is symmetric with respect to the reflection $\gs _e $, 
 \eqref{EquazDifferenzaSistemi}  is fully coupled in $\gO (e)$ as well and i) is proved.\par
 To get \eqref{UgPartiPosNeg-Q-e} it is enough  to multiply the $i$-th equation of the system for $w_i$ and integrate. 
 Instead, multiplying the $i$-th equation of  \eqref{EquazDifferenzaSistemi} for $w_i^+$, we get  
   $$
   0= \int _{\gO (e)} ( |\nabla w_i^+ |^2 + \sum _{j=1}^m b_{ij}^{e} w_j w_i^+ )\, dx  \geq 
\int _{\gO (e)} (|\nabla w_i ^+ |^2 + \sum _{j=1}^m b_{ij}^{e} w_j^+ w_i^+ )\, dx 
$$
 since $w_i \, w_i ^+ = |w_i ^+|^2$, while $w_j w_i^+ \leq w_j^+ w_i^+$ and $b_{ij} \leq 0 $ if $i \neq j $. \\
Summing on $i$ we get  
$$
0   \geq 
\int _{\gO (e)} \sum _{i=1}^m |\nabla w_i^+ |^2 + \sum _{i,j=1}^m b_{ij}^{e} w_j^+ w_i^+ \, dx 
$$  
 i.e. \eqref{DisugPartiPosNeg-Q-e} in the case of the positive part.     \par           
For the negative part we proceed analogously
multiplying the $i$-th equation of  \eqref{EquazDifferenzaSistemi} for $w_i^-$ and integrating.  We get \par
\smallskip
  $0=  - \int _{\gO (e)} |\nabla w_i^- |^2 + \sum _{j=1}^m b_{ij}^{e} w_j w_i^- \, dx  \leq 
-\int _{\gO (e)} |\nabla w_i ^-|^2 + \sum _{j=1}^m b_{ij}^{e} (-w_j^-) w_i^- \, dx $ \par
\smallskip
$= -\int _{\gO (e)} |\nabla w_i ^-|^2 - \sum _{j=1}^m b_{ij}^{e} (w_j^-) w_i^- \, dx $ \par
\smallskip 
\noi since $w_i \, w_i ^- = - |w_i^-|^2$, while $w_j w_i^- \geq -(w_j^-) w_i^-$ and $b_{ij} \leq 0 $ if $i \neq j $. \\
Summing on $i$ we obtain
$$
0   \geq 
\int _{\gO (e)} \sum _{i=1}^m |\nabla w_i ^-|^2 + \sum _{i,j=1}^m b_{ij}^{e} w_j^- w_i^- \, dx 
$$
  i.e. \eqref{DisugPartiPosNeg-Q-e} in the case of the negative part.
\par

\end{proof}

\begin{remark} \  
 Note that  the inequalities in  \eqref{DisugPartiPosNeg-Q-e} could be strict. Indeed the products  $w_i^+ \, w_j ^-$ could  be not identically zero  if $i \neq j$, and therefore 
 $Q(W^e)$ does not coincide in general with $Q((W^e)^+) + Q ((W^e)^-)$, as it happens in the scalar case.
 \end{remark}
 
\subsection{Nonlinearities having convex components}  \  \\
We will prove Theorem \ref{fconvessaSistemi} by  several auxiliary  results.
\par
\smallskip


 \begin{lemma}\label{lemma4}  Assume that $U$ is a solution of \eqref{modprob}  and that the hypotheses i)--ii) 
 of Theorem \ref{fconvessaSistemi} hold. 
Then   for any direction $e \in S^{k-1}$
$$Q_U\left ( (W^{e})^{+}; \gO (e)  \right )  \leq 0 
$$
 where $Q_U$ is the quadratic form defined in \eqref{Xformaquadraticalinearizzato} and $W^e $ is as in Lemma \ref{EqDifferSistemi}.

 \end{lemma}
 
 \begin{proof}
 For any $i=1, \dots , m $ we have  
  $$
  - \gD w_i= f_i(|x|,U)-f_i(|x|,U^{\gs (e)})  \quad \text{ in } \gO (e)
  $$
 \smallskip
 \noi Testing the equation with $w_i^+$ we obtain 
 \be \label{ProvvEqDiff}
 \int _{\gO (e)}    | \nabla (w_i )^{+}|^2 \, dx  = \int _{\gO (e)} \left ( \, f_i(|x|,U)-f_i(|x|,U^{\gs (e)})  \, \right ) \,  w_i^+   \, dx 
 \ee
  Observe that     $f_i(|x|,S)$ is convex in $S$, so that \par
  \smallskip
$ ( \, f_i(|x|,U(x))-f_i(|x|,U^{\gs (e)}(x)) \, )  w_i^+  \leq ( \, \nabla f_i(|x|,U(x)) \cdot (U(x)- U^{\gs (e)}(x)) \, )  w_i^+ $ \par
\smallskip
$= ( \, \nabla f_i(|x|,U(x)) \cdot W^{e} \, )  w_i^+ 
= \sum _{j=1}^m  \frac {\de  f_i}{\de u_j}(|x|,U(x)) \,w_j \,    w_i^+  $ \par
\smallskip
\noi where 
$\nabla $ stands for the gradient of $f_i$ with respect to the variables $S=(s_1, \dots , s_m)$. Moreover   
$$ \frac {\de  f_i}{\de s_i} w_i \, w_i^+ = \frac {\de  f_i}{\de s_i} |w_i^+|^2 \; , \;  \text{ while } \; 
   \frac {\de  f_i}{\de s_j} w_j \, w_i^+  \leq \frac {\de  f_i}{\de s_j} w_j^+ \, w_i^+ \;  \text{ if } 
i \neq j
$$
  because $\frac {\de  f_i}{\de s_j} \geq 0 $ by the weak coupling assumption.\\
By \eqref{ProvvEqDiff}, taking into account the previous inequalities, we get
$$
 \int _{\gO (e)}    | \nabla (w_i )^{+}|^2 \, dx  \leq  \int _{\gO (e)} \sum _{j=1}^m  \frac {\de  f_i}{\de s_j}(|x|,U(x)) \,w_j ^+ \,    w_i^+     \, dx 
$$
  Thus, summing on $i= 1, \dots , m$, we  obtain

\be 
\int _{\gO (e)}    \left ( \, \sum _{i=1}^m  | \nabla (w_i )^{+}|^2 - \sum _{i,j=1}^m  \frac {\de  f_i}{\de s_j}(|x|,U(x)) \,   w_i^+ \, w_j^+ \, \right ) \, dx \leq 0
\ee
i.e.  $Q_U\left ( (W^{e})^{+}; \gO (e)  \right )  \leq 0 $.
  \end{proof}
  \par
  \medskip
  
  If $e \in S^{k-1}$ is a direction  orthogonal to $\bold e_{k+1}, \dots , \bold e_N$ and $C=\frac 12 \left (  J_F(x)+ J_F(x)^t \right )$, let us denote 
     the eigenvalues and the eigenfunctions of the operator $- \Delta - C $ in the cap $\Omega (e)$   by 
     \be \label{DefinAutovCappeLinear}  \gl _k^e= \gl _k (- \Delta - C \; ; \; \Omega (e)) \quad ; \quad  
     \gF _k^e  = \gF _k (- \Delta - C \; ; \; \Omega (e))
     \ee 

\begin{lemma}\label{lemma5}  Suppose that $U$ is a solution of \eqref{modprob} with Morse index $m (U) \leq k$ and assume that the hypothesis i) of Theorem \ref{fconvessaSistemi} holds.
Then there exists a direction $e \in S^{k-1}$ such that  \ \ $ \gl _1^e     \geq 0 $, 
 hence also the corresponding principal eigenvalue $\tilde {\gl }_1 (L_U ,\gO (e) )$ is  nonnegative, by Theorem \ref{principaleigenvalue}, so that 
$$
 Q_U (\Psi ; \gO (e) )  \geq 0
  \notag
$$
for any $\Psi \in C_c^1 (\gO (e);\R^m)$ .\\
\end{lemma}
\begin{proof}   
The assertion is immediate if the Morse index of the solution satisfies $m (u) \leq 1$. 
 Indeed in this case for any direction $ e$ at least one among
$ \gl _1^{e}$  and 
$ \gl _1^{-e}$ must be nonnegative.
 Indeed if this would not be the case then the quadratic form 
 $Q_U (\Psi )= \int _{ \Omega } ( |\nabla \Psi |^2 - C(x) \, (\Psi , \Psi) ) \, dx $ would be negative definite on the $2$- dimensional space spanned by the trivial extensions of the eigenfunctions
 $\gF _1^e $ and $\gF _1^{-e} $ and hence $m(u) \geq 2$. 
\par
So let us assume that   $2\leq j= m (u) \leq k$.\par
Denote by   $\Phi_k$ the $L^2 (\Omega)$ normalized eigenfunctions of the operator $L_U=- \Delta - C $ in $\Omega $, with  $\Phi _1 $ positive in $\Omega$,  and 
for any direction $ e \in S^{k-1}$
let us consider the function 
$$ \gPs ^{ e}(x)=
 \begin{cases} 
 \left (  \frac  {  ( \gF _1^{-e} \,, \, \Phi _1 )_{L^2(\Omega )} }  {  ( \gF _1^{e} \,, \, \Phi _1 )_{L^2(\Omega )} }  \right ) ^{\frac 12} \gF _1^{e} (x)
  \quad &\text{ if } x \in \gO (e)    \\
  -  \left (  \frac  {  ( \gF _1^{e} \,, \, \Phi _1 )_{L^2(\Omega )} }  {  ( \gF _1^{-e} \,, \, \Phi _1 )_{L^2(\Omega )} }  \right ) ^{\frac 12} \gF _1^{-e} (x)
  \quad &\text{ if } x \in \gO (-e)    
\end{cases}
$$
where $ \gF _1^{e} $ is the first positive $L^2$-normalyzed eigenfunction in $\Omega (e)$, as in \eqref{DefinAutovCappeLinear}.
  \par
 The mapping $e \mapsto \gPs _e $ is odd and  continuous    from $S^{k-1}$ to $H_0^1 (\Omega ) $  and, by construction,  
 \be (\gPs ^{ e} \,, \, \Phi _1 )_{ L^2(\Omega )} =0
 \ee
The function $h: S^{k-1}\to \R ^{j-1} $ defined by
 \be h( e)= \left ( ( \gPs ^{ e} \,, \, \Phi _2 )_{L^2(\Omega )},   \dots , ( \gPs ^{ e} \,, \, \Phi _j )_{L^2(\Omega )} \right )
 \ee
 is also odd and  continuous.  Since $(j-1 ) < k $,  by the Borsuk-Ulam Theorem it must  have a zero. 
 This means that there exists a direction $e\in S^{k-1}$ such that
   $\gPs ^{ e} $ is orthogonal to all the eigenfunctions $\Phi _1, \dots , \Phi_j$. Since $ m (u)=j $, by Theorem \ref{Xvarformautov} \  ii)   we deduce  that 
 $Q_U(\gPs ^{ e} ; \gO ) \geq 0 $, which in turn implies that either $Q_u(\gF _1 ^{ e} ; \gO (e) ) \geq 0 $ or 
 $Q_U(\gF _1^{ -e } ; \gO (-e) ) \geq 0 $, i.e. either
 $\gl _1^{ e} $ or $\gl _1^{- e} $ is nonnegative, so the assertion is proved.
\end{proof}
 \par 
 \smallskip 
\begin{proof}[Proof of Theorem \ref{fconvessaSistemi}]
By Lemma \ref{lemma5} there exists a direction $e \in S^{k-1}$ such that the first symmetric eigenvalue $\gl _1^{\text{s}} (L_U, \gO (e))$ of the linearized operator 
  is nonnegative, so that the principal eigenvalue $\tilde {\gl}_1 (\gO (e))$ is nonnegative as well.
 Moreover by Lemma \ref{lemma4} we have that $Q_U\left ( (W^{e} )^{+}  \right ) \leq 0 $, so that either $ (W^{e} )^{+} \equiv 0 $, or $\gl _1^{\text{s}} (L_U, \gO (e))=0 $ and  $ ( W^{e} )^{+}$ is the positive  first symmetric eigenfunction in $\gO (e)$. 
  In any case  either $U\leq U^{\gs (e)}$ or 
 $U\geq U^{\gs (e)}$ in $\gO (e)$ holds.\\
Thus, by the strong maximum principle, either  $U \equiv  U^{\gs (e)} $ in $\gO (e)$, and the principal eigenvalue $\tilde {\gl}_1 (\gO (e))$  is nonnegative,  
or $U <  U^{\gs (e)} $ in $\gO (e)$ or $U >  U^{\gs (e)} $ in $\gO (e)$. 
Hence, by Theorem \ref{SuffCondSectFSS2Sistemi} 
  $U$ is foliated Schwarz symmetric. \par
   \end{proof}
\par
\medskip
\begin{remark} \label{AutovPrincZero} \ 
  In the previous proof when $ (U-U^{\gs (e)})^{+} \equiv 0 $ we also have 
  by construction that $\gl _1^{\text{s}} (L_u, \gO (e)) \geq 0 $ and therefore the principal eigenvalue satisfies \  $  \tilde {\gl} _1 (\gO (e)) = \tilde {\gl} _1 (\gO (-e))  \geq 0 $. \par    
   In the case when $U <  U^{\gs (e)}$ in $\gO (e)$ or $U >  U^{\gs (e)}$ in $\gO (e)$, by rotating the planes we find a different direction $e'$ such that 
   $U \equiv U^{\gs (e')} $ in $\Omega (e')$ and it could happen that
 $ \gl _1^{\text{s}} (\gO (e')) <0 $. 
 However let us observe
    explicitly  that  the sign of the principal eigenvalue is preserved in the rotation, i.e.
  $ \tilde {\gl} _1 (\gO (e')) = \tilde {\gl} _1 (\gO (-e'))  \geq 0
  $, and actually  $ \tilde {\gl} _1 (\gO (e')) = \tilde {\gl} _1 (\gO (-e'))  = 0 $. \par
    Indeed since $U < U^{\gs (g)}$ for any direction $g$  between  $e$ and $e'$, we have that $0$ is the principal eigenvalue of the system satisfied by $U- U^{\gs (g)}$, namely \eqref{EquazDifferenzaSistemi}, with coefficients 
   $$
 b_{ij}^{g}(x) = - \,  \int _0^1 \frac {\de f _i} {\de s_j} \left [|x|, tU(x)+ (1-t)U^{\gs (g)}(x)  \right ] \, dt  
$$ 
As $g \to e' $, where $e'$ is the symmetry position,  the coefficients $b_{ij}$ approach the coefficients of the linearized system, namely
$c_{ij}= - \,  \frac {\partial f_i} {\partial s_j}$, so by continuity $ \tilde {\gl} _1 (\gO (e')) = \tilde {\gl} _1 (\gO (-e'))  = 0 $.\par
       \end{remark}
\par
\smallskip

\subsection{Nonlinearities with  convex derivatives}  \ \\      
The proof of Theorem \ref{f'convessaSistemi} follows the scheme of the proof of Theorem  \ref{fconvessaSistemi},  and it is based upon the following results.
 \par
\smallskip
 
    \begin{lemma}\label{Xlemma1} 
   Assume that $U$ is a solution of \eqref{modprob} and the hypotheses  of Theorem \ref{f'convessaSistemi} hold.
   Let $\; B^{e}(x) =\left (b_{ij}^{e}(x) \right )_{i,j=1}^m $ be the matrix associated to the fully coupled system \eqref{EquazDifferenzaSistemi} defined by \eqref{CoeffEquazDifferenzaSistemi}, i.e. 
   \be\nonumber  
 b_{ij}^{e}(x) = - \int _0^1 \frac {\de f _i} {\de s_j} \left [|x|, tU(x)+ (1-t)U^{\gs (e)}(x)  \right ] \, dt  
\ee 
and let us define the matrix 
  $\quad  B^{e,s}(x) =\left (b_{ij}^{e,s}(x) \right )_{i,j=1}^m  \quad $, where
\be \label{DefCoeffSimm} b_{ij}^{e,s} (x) =  - \frac 12 \left(  \frac {\de f _i} {\de s_j} (|x|,U(x)) +  \frac {\de f _i} {\de s_j} (|x|,U^{\gs  (e)}(x)) \right ) 
\ee
Then the linear system with matrix $B^{e,s} $
is fully coupled in $\gO $ and $\gO (e)$ for any  $e\in S^{N-1}$. Moreover 
  for any  $i,j=1,\dots ,m$  and $ x \in \gO $  it holds 
\be \label{CfrCoefficienti} b_{ij}^{e} (x)\;  \geq \; b_{ij}^{e,s} (x)  
\ee
 \noi  Finally  for the  quadratic forms $ Q^{e} $ and $ Q^{e,s} $  associated to the matrices $ B^{e} $ and $ B^{e,s} $ we have that
\begin{multline} \label {cfrformequadratiche}
 0 \geq  Q^{e} (\,  (W^{e})^{\pm}; \gO (e)  \,) = \int _{\gO (e)} \left [ \,  |\nabla (W^{e})^{\pm} |^2 + B^{e}(\,  (W^{e})^{\pm} ,(W^{e})^{\pm}  \,) \, \right ] dx  \\
\geq  \int _{\gO (e)} \left [ \, |\nabla (W^{e})^{\pm} |^2 + B^{e,s}(\,(W^{e})^{\pm} ,(W^{e})^{\pm} \,  ) \, \right ] dx  = Q^{e,s}( \, (W^{e})^{\pm}; \gO (e) \,) 
\end{multline}
for $W^{e}= U - U^{\gs (e)} $, 
\end{lemma}

\begin{proof}  
By hypothesis ii)  of Theorem \ref{f'convessaSistemi}  we get
\begin{multline} - b_{ij}^{e}(x) = \int _0^1   \frac {\de f _i} {\de s_j} \left [|x|, tU(x)+ (1-t)U^{\gs (e)}(x)  \right ] dt \\
\leq 
\int _0^1 \left (   t\, \frac {\de f _i} {\de s_j} \left [ |x|, U(x) \right ] + (1-t) \frac {\de f _i} {\de s_j} \left [ |x|,U^{\gs  (e)}(x)  \right ]  \right )dt \\
=
 \frac 12 \left(  \frac {\de f _i} {\de s_j} (|x|,U(x)) +  \frac {\de f _i} {\de s_j} (|x|,U^{\gs  (e)}(x)) \right )= - b_{ij}^{e,s}(x) 
 \end{multline}
 This implies \eqref{CfrCoefficienti}  and hence the
full coupling of the system with matrix $B^{e,s} $, since, by Lemma \ref{EqDifferSistemi},
the system with matrix $B^{e} $ is fully coupled.  
From \eqref{DisugPartiPosNeg-Q-e}  and \eqref{CfrCoefficienti}, since $w_k^{\pm} \geq 0 $, we get 
\be \nonumber
\begin{split}
0   \geq 
& \int _{\gO (e)} \left ( \sum _{i=1}^m |\nabla w_i^+ |^2 + \sum _{i,j=1}^m b_{ij}^{e} w_j^+ w_i^+  \right )\, dx \geq  \\
& \int _{\gO (e)} \left ( \sum _{i=1}^m |\nabla w_i^+ |^2 + \sum _{i,j=1}^m b_{ij}^{e,s} w_j^+ w_i^+  \right ) \, dx 
\end{split}
\ee
 i.e. \eqref{cfrformequadratiche} in the case of the positive part of $W^{e} $.
Analogously we get the corresponding inequality for the negative part of $W^{e} $.
    \end{proof} 
  \par
\smallskip 

\begin{lemma}\label{Xlemma4}  Suppose that $U$ is a solution of \eqref{modprob} with Morse index $m (U) \leq k-1$ and assume that the hypothesis i) of Theorem \ref{f'convessaSistemi} holds. \\
Let  $Q^{e,s}$ be 
the quadratic form associated to the operator $L^{e,s} (V) = - \gD V + B^{e,s}V $, \ $B^{e,s}$ being the matrix defined in \eqref{DefCoeffSimm} :
\be \label{formaquadraticasimmetricalinearizzato}
\begin{split}
&Q^{e,s} (\Psi ; \gO ' ) = \int _{\gO '} \left [ |\nabla \Psi |^2 + B^{e,s}(\Psi ,\Psi )\right ] dx= \\
& \int _{\gO}\left [   \sum _{i=1}^m |\nabla \psi _i |^2 -\sum _{i,j=1}^m  \frac 12 \left ( \frac {\de f_i}{\de s_j}(|x|, U(x))+  \frac {\de f_i}{\de s_j}(|x|, U^{\gs  (e)}(x)) \right ) \psi _i \psi _j  \right ]  \, dx 
\end{split}
\ee
Then there exists a direction $e \in S^{k-1}$ such that 
$$
 Q^{e,s} (\Psi ; \gO (e) )  \geq 0 \quad \quad  \forall \; \Psi  \in C_c^1 (\gO (e);\R^m)
  \notag
$$
Equivalently  
the first \emph{symmetric} eigenvalue $\gl _1^{\text{s}} (L^{e,s}, \gO (e))$ of the  operator $L^{e,s} (V) = - \gD V + B^{e,s}V $ in $\gO (e)$ is nonnegative (and hence also the principal eigenvalue $\tilde {\gl }_1 (L^{e,s} ,\gO (e) )$ is  nonnegative).
\end{lemma}

\begin{proof} 
Let us assume that  $1\leq j= m (U) \leq k-1 $ and let $\Phi _1, \dots , \Phi _j$ be mutually orthogonal eigenfunctions corresponding to the negative symmetric eigenvalues $\gl _1^{\text{s}} (L_U, \gO )$, \dots , $\gl _j^{\text{s}} (L_U, \gO )$ of the linearized operator $L_U (V) = - \gD V - J_F (x, U) V $ in $\gO $ .\\
 For any $e\in S^{k-1}$ let $\phi ^{e,s}$ be the first positive $L^2$-normalized eigenfunction  of the symmetric system associated to the linear operator 
 $L^{e,s}  $ in $\gO (e)$. We observe that $\phi ^{e,s}$ is uniquely determined since the corresponding system is fully coupled in $\gO (e)$.
 Let $\Phi ^{e,s}$ be the odd extension of $\Phi ^{e,s}$ to $\gO$, and let us observe that  $\Phi ^{-e,s} = - \Phi ^{e,s}$, because $B^{e,s}$ is symmetric with respect to the reflection $\gs _e$. \\
 The mapping $e \mapsto \Phi ^{e,s} $ is a continuous odd function from $S^{k-1}$ to $\huno $,    
  therefore the mapping $h: S^{k-1}\to \R ^{j} $ defined by
 $$ h(e)= \left ( (\Phi ^{e,s} \, , \, \Phi _1 )_{\ldue (\gO )},   \dots , ( \Phi ^{e,s} \,, \, \Phi _j )_{\ldue (\gO )} \right )
 $$
 is an odd continuous mapping, and since $j \leq  k-1 $,  by the Borsuk-Ulam Theorem it must  have a zero.
 This means that there exists a direction $e\in S^{k-1}$ such that
   $\Phi ^{e,s}  $ is orthogonal to all the eigenfunctions $\Phi _1, \dots , \Phi _j$.   This implies that 
 $Q_U(\Phi ^{e,s}; \gO )  \geq 0 $, because $ m (U)=j $, and since $\Phi ^{e,s}$ is an odd function, we obtain that 
 $0 \leq Q_U(\Phi ^{e,s} ; \gO )  =  Q^{e,s} (\Phi ^{e,s}, \gO )= 2 Q^{e,s} (\phi ^{e,s}, \gO (e)) = 2  \gl _1^{\text{s}} (L^{e,s}, \gO (e)) $
 \end{proof}
 \par
\medskip

\begin{proof}[Proof of Theorem \ref {f'convessaSistemi}] 
By Lemma \ref{Xlemma4} there exists a direction $e$ such that 
the first symmetric eigenvalue $\gl _1^{\text{s}} (L^{e,s}, \gO (e))$ of the  operator $L^{e,s} (V) = - \gD V + B^{e,s}V $ in $\gO (e)$ is nonnegative, and hence also the principal eigenvalue $\tilde {\gl }_1 (L^{e,s} ,\gO (e) )$ is  nonnegative.\par 
Since $Q^{e,s}( \, (W^{e})^{\pm}; \gO (e) \,)  \leq 0 $ by Lemma \ref{Xlemma1}, we have two alternatives.
The first one is that  $(W^{e})^{+}$ and $(W^{e})^{-}$ both vanish, in which case $W^{e} \equiv 0 $ in $\Omega (e)$, and this implies in turn that $L^{e,s}= L_U$. Then 
$U $ is symmetric  and  the principal eigenvalue $\tilde {\gl }_1 (L_{U} ,\gO (e) )$ is  nonnegative, so that 
the hypothesis i) of Theorem  \ref{SuffCondSectFSS2Sistemi}  holds and we get that  $U$ is  
 foliated Schwarz symmetric.
 The second alternative  is that one among  $(W^{e})^{+}$ and $(W^{e})^{-}$ does not vanish and  $\gl _1^{\text{s}} (L^{e,s}, \gO (e))=0$. Then either  
 $(W^{e})^{+}$ or $(W^{e})^{-}$ is a first symmetric eigenfunction of the  operator $L^{e,s} (V)  $ in $\gO (e)$. If  
 $(W^{e})^{+}$ is a  first symmetric eigenfunction of the  operator $L^{e,s} (V) = - \gD V + B^{e,s}V $ in $\gO (e)$ then it is positive  in $\gO (e)$, i.e.
    $U > U^{\gs _{e}} $  in $\gO (e)$. In the case when $(W^{e})^{-}$ is the  first symmetric eigenfunction we get the reversed inequality.  
Then, by the sufficient condition ii) given by Theorem \ref{SuffCondSectFSS2Sistemi},  $u$ is foliated Schwarz symmetric. \par
 \end{proof}
 \par
 \medskip
 \begin{proof}[Proof of  Corollary \ref{corollario1}]  \      
By the  proof of Theorem \ref{fconvessaSistemi} and Remark \ref{AutovPrincZero}   we can find  $e \in S^{k-1}$  such that $U$ is symmetric with respect to the hyperplane $H(e)$ and the principal eigenvalue 
$ \tilde {\gl} _1 (\gO (e)) = \tilde {\gl} _1 (\gO (-e))  \geq 0 $.
As in the proof of Lemma \ref{lemma5}  it is easy to see that
if  $U$ is a Morse index one solution then for any direction $e$ either   $\gl _1^{\text{s}} (L_U, \gO (e))$ or $\gl _1^{\text{s}} (L_U, \gO (-e))$ must be nonnegative. 
On the other hand, by  symmetry, $ \gl _1 ^{\text{s}} (\gO (e)) = \gl _1 ^{\text{s}} (\gO (-e)) $, so that
$ \tilde {\gl} _1 (\gO (e)) = \tilde {\gl} _1 (\gO (-e))  \geq  \gl _1 ^{\text{s}} (\gO (e)) = \gl _1 ^{\text{s}} (\gO (-e))   \geq 0 $.\par
Then, if $ \tilde {\gl} _1 (\gO (e)) >0 $,  the angular derivative $U_{\theta} $  must vanish (since it satisfies \eqref{SistemaLinearizzatoCappa} and  the maximum principle holds in 
$ \gO (e)$). Hence 
$U$ is radial. \\
So if  $U_{\theta} \not \equiv 0 $ necessarily   $ \tilde {\gl} _1 (\gO (e)) = \gl _1 ^{\text{s}} (\gO (e))  =0 $ and by iv) of Proposition \ref{principaleigenvalue}  the derivative
$U_{\theta}$ is a negative first eigenfunction of the simmetrized system in $\gO (e)$, as well as a solution of \eqref{SistemaLinearizzatoCappa}. 
Thus  we get that 
$$J_F (|x|, U) U_{\theta}= \frac 12 \left (J_F (|x|, U) + J_F ^t (|x|, U) \right ) U_{\theta}
$$
 i.e. 
  \eqref{superfullycoupling1} and if $m=2$ we get    \eqref{superfullycoupling2},  since $U_{\theta}$ is strictly  negative. \par
  The proof in the case when the hypotheses of Theorem \ref{f'convessaSistemi} hold is the same.  
  \end{proof}
 \par
\medskip

\end{document}